\documentclass[runningheads,twoside]{llncs}

\usepackage[ansinew]{inputenc}
\usepackage[T1]{fontenc}

\usepackage{multirow,bigdelim}   
\usepackage{graphicx}
\usepackage{amssymb}
\usepackage{textcomp}
\usepackage{amscd}
\usepackage{amsmath}
\usepackage{hyperref}
\usepackage{algorithm,algorithmic}

\begin{document}

\mainmatter              
%opening
\title{The isomorphic version of Brualdi's and Sanderson's nestedness is in P}

\titlerunning{Isomorphic Nestedness}
\author{Annabell Berger}
\institute{Institute of Computer Science, Martin Luther University Halle-Wittenberg,\\ Halle (Saale), Germany\\\href{mailto:berger@informatik.uni-halle.de}{berger@informatik.uni-halle.de}}
%\email{annabell.berger@informatik.uni-halle.de}
\authorrunning{A. Berger}
\maketitle

\begin{abstract}
The discrepancy BR for an $m \times n$ $0,1$-matrix from Brualdi and Sanderson 1998 counts the minimum number of $1$'s which need to be shifted in each row to the left to achieve its Ferrers matrix, i.e. each row consists of consecutive $1$'s followed by consecutive $0$'s. For ecological bipartite networks BR describes how nested a set of relationships is. Since different labeled matrices can be isomorphic but possess different discrepancies, we define a metric determining the minimum discrepancy in an isomorphic class. We give a reduction to $k\leq n$ minimum weighted perfect matching problems.
\end{abstract}

\section{Introduction}

Let's start with a formal description. Given is an $m \times n$ $0,1$-matrix $A$ with non-increasing row sums $(r_1,\dots,r_m)$ and non-increasing column sums $(c_1,\dots,c_{n}).$ The corresponding \emph{Ferrers matrix} $F$ is the matrix where each row $i$ of $F$ starts with $r_i$ $1$'s followed by $n-r_i$ $0$'s. It can be achieved by several \emph{shifts}. A shift is the movement of a $1$ in a row from a righter to a lefter column with entry $0$. Ferrers matrices together with majorization theory are classical concepts for deciding the question if there exists a matrix $A$ with given row and column sums. They were invented by Ferrers \cite{Syl:1882} in the $18$th century. Matrix $F$ has also row sums $(r_1,\dots,r_n)$ but different column sums, i.e., $(c'_1,\dots,c'_n)$ where $c'_i=|\{j~|~j \textnormal{ is row in A with } r_j \geq i\}|.$ A matrix $A$ can also be interpreted as the adjacency matrix of a bipartite graph, and the corresponding graph of $F$ is then a \emph{threshold graph}. For an overview about these problems we recommend Brualdi's book \cite{Brualdi2006} and Mahadev's and Peled's book \cite{MaPe:95}.\\
However, Ferrers matrices also occur in ecology under the name \emph{nested matrices} to describe idealized ecological networks due to the strength of nestedness. In these contexts matrix $A$ represents the presence or absence of interactions between pollinator and plant species, or, the occurrence and absence of species on several islands. If $A$ was a nested matrix, then one bee species in an arbitrary pair of bee species, would always pollinate a subset of the plant species which is pollinated from the other one. 
The opposite scenario is that, all bee pairs pollinate completely different plants. If $A$ was not nested the natural question arises of how far the ecological matrix is apart from being nested. Brualdi and Sanderson answered this question by the metrics of discrepancy $disc(A)$ between $A$ and $F$ \cite{Brualdi1999}. They count how many $1$'s need to be shifted at least to the left to achieve $F$. This value can easily be calculated but there is one conceptual problem regarding the ecological application. The exchange of two columns in $A$ with same column sum keeps all interactions between pollinators and bees but can lead to a different discrepancy. The reason is that both matrices are formally considered as different because two bees have different labels. Consider Example~\ref{Example:DifferentDiscrepancies}. 

\begin{example}
\label{Example:DifferentDiscrepancies}
Matrix $A:=\left(\begin{matrix} 1&\mathbf{0}&1&\mathbf{1}\\1&\mathbf{0}&1&\mathbf{1}\\1&1&\mathbf{0}&\mathbf{1}\\1&1&0&0\\1&1&0&0\\1&\mathbf{0}&\mathbf{1}&0\end{matrix}\right )$ and  
$B:=\left(\begin{matrix} 1&\mathbf{0}&1&\mathbf{1}\\1&\mathbf{0}&1&\mathbf{1}\\1&1&1&0\\1&1&0&0\\1&1&0&0\\1&\mathbf{0}&0&\mathbf{1}\end{matrix}\right )$ are isomorphic. Only the third and fourth columns are exchanged. In $A$ we need to shift the four bold $1$'s to achieve its corresponding Ferrers matrix $F=\left(\begin{matrix} 1&1&1&0\\1&1&1&0\\1&1&1&0\\1&1&0&0\\1&1&0&0\\1&1&0&0\end{matrix}\right )$, i.e. $disc(A)=4.$ In matrix $B$ we only need to shift three $1$'s to achieve $F$ leading to $disc(B)=3.$ 
\end{example}

For ecological purposes these labels are not necessary. They are only assigned to columns and rows due to the mathematical denotation of columns and rows. In more formal words both matrices are isomorphic but lead to several values of discrepancy. The discrepancy between $A$ and $F$ should be extended to the problem of finding the minimal number of shifts under all isomorphic graphs of $A$ to achieve $F$. The \emph{isomorphic discrepancy} finds an optimal labeling for a network regarding a minimum number of shifts to achieve the nested matrix $F.$ We solve this problem in reducing it to a couple of maximum perfect weighted matching problems leading to an $O(n^3)$ asymptotic running time.\\
Please observe that there exist other approaches and assumptions how a perfect nested matrix $N$ for $A$ has to look like. This results in another difference of both matrices. One approach is to replace a minimal number of $0$'s by $1$'s in $A$ such that $A$ (or an isomorphic matrix of $A$) becomes nested. This problem is known as \emph{minimal chain completion problem} and was shown to be NP-complete in 1981 by Yannakakis \cite{Yannakakis1981}. The opposite problem is to delete a minimal number of $0$'s in $A$ to achieve a nested matrix which is also NP-complete (please exchange the rules of $1$'s and $0$'s to see this). An approximation algorithm was given in the Phd-thesis of Juntilla in 2011 \cite{Junttila2011}. A whole bunch of formulas are used in the ecological community. For an overview see the nestedness guide \cite{Ulrich2009} which also includes the discrepancy of Brualdi and Sanderson (there called BR). All these metrics can be seen as heuristics for the above described approaches.\\
In the next section we give a formal definition of \emph{isomorphic discrepancy}, and show how it can be solved in asymptotic running time $O(n^3).$

\section{The discrepancy problem for isomorphic matrices} 

Let $A$ be an $m \times n$ $0,1$-matrix with non-increasing row sums $(r_1,\dots,r_m)$, and non-increasing column sums $(c_1,\dots,c_{n}).$ Then all isomorphic matrices of $A$ can be achieved by permuting columns and rows with equal sums. We denote the set of all isomorphic matrices of A by $I_A.$ The discrepancy of A given by Brualdi and Sanderson \cite{Brualdi1998} is the minimum number of shifts to achieve the corresponding Ferrers matrix $F$ as described in the last section. We denote the discrepancy of A by $disc(A)$.
We now give the definition of the \emph{isomorphic discrepancy} $Id(A)$ of A which is defined by 
\begin{equation} Id(A):=min_{B \in I_A}disc(B).\end{equation}
That is we try to find an isomorphic graph of $A$ with minimal discrepancy. We also want to mention that this new metric should be invariant against the transposition of $A.$ This can be achieved by transposing $A$ and determine $Id(A^t).$ The \emph{general isomorphic discrepancy} of $A$ is then the minimum value (or the mean value) of $Id(A)$ and $Id(A^t).$ We here focus on $Id(A)$, because the calculation of $Id(A^t)$ can be done analogously. \\
A simple approach were to permute all columns and rows with equal sums and to choose a permutation of $A$ which yields the minimum discrepancy. This can lead to the determination of exponentially many possible permutations. Please observe that it is sufficient for our problem to permute columns in $A$ and keep the rows fixed. The reason is that exchanging the order of rows with same sums lead to exactly the same discrepancy because the corresponding rows of the Ferrers matrix are identical. Consider our Example~\ref{Example:DifferentDiscrepancies} it can easily be seen that each permutation of the first three rows in $A$ or $B$ leads to the same kind of shifts, and so to the same discrepancy as before. Following this approach we will not necessarily generate all isomorphic matrices in $I_A$ but we will generate all possible discrepancies which occur for matrices in $I_A.$ The reason is that the switching of rows with equal sums does not change the number of shifts in a row. Hence, \emph{it is sufficient to permute columns for the calculation of $Id(A).$} Brualdi and Shen showed \cite{Brualdi1999} that there is always one matrix $C$ with minimum possible discrepancy $disc(C)=\sum_{j=1}^n(c'_j-c_j)^+$ under all matrices with row sums $(r_1,\dots,r_m)$ and column sums $(c_1,\dots,c_{n}).$  Example~\ref{Example:MinimumExample} shows that not each isomorphic class $I_A$ possesses a matrix $C$ with minimum discrepancy.

\begin{example}
\label{Example:MinimumExample}
For matrix $M:=\left(\begin{matrix} 1&\mathbf{0}&1&\mathbf{1}\\1&1&1&0\\1&1&\mathbf{0}&\mathbf{1}\\1&1&0&0\\1&\mathbf{0}&0&\mathbf{1}\\1&\mathbf{0}&\mathbf{1}&0\end{matrix}\right )$ we have $disc(M)=4$, and corresponding Ferrers matrix $F$ from Example~\ref{Example:DifferentDiscrepancies}. The column sums from $M$ are $(6,3,3,3)$ and for $F$ we find $(6,6,3,0).$  Comparing the two column sums $(6,3,3,3)$ and $(6,6,3,0)$ one could get the impression that it could be sufficient to shift three $1$'s of the fourth column to the second column. In $M$ this is not possible because the third columns of $M$ and $F$ differ in the third and sixth rows. Hence, a $1$ in the third column of $M$ needs to be shifted to generate $F$. On the other hand it is easy to construct a matrix $C$ with $disc(C)=\sum_{j=1}^n(c'_j-c_j)^+=3$ by shifting the $1$'s of the first, fourth and sixth row in the second column of $F$ to the fourth column. However, $Id(M)=4$ because each permutation $\sigma$ of the last three columns of $M$ leads to a matrix $M^{\sigma}$ with $disc(M^{\sigma})=4.$ We observe in the third column of an $M^{\sigma}$ that the third column misses a $1$ whereas $F$ has one. Matrix $C$ cannot be achieved by a permutation of columns from $M$, i.e., $C \notin I_M.$ 
\end{example}

Now we want to develop a simple formula to calculate the discrepancy of a given matrix. We need this formula later to deviate an approach for the isomorphic version. Notice that a shift in a matrix $A$ will always be applied when $F_{i,j}=1$ and $A_{i,j}=0.$ Since $A$ and $F$ have identical row sums and due to the construction of $F$ there must always be an index $j'>j$ in $A$ with $F_{i,j'}=0$ and $A_{i,j'}=0.$ In our Examples~\ref{Example:DifferentDiscrepancies}, \ref{Example:MinimumExample} we marked these $0$'s in $A$, $B$ and $M.$ That means that the number of absences of $1$'s in $A,$ when they are present in $F,$ correspond to the number of minimum shifts $disc(A)$. We define the difference of the $j$th columns $V_j$ and $U_j$ of $m \times n$ $0,1$-matrices $V$ and $U$ by

\begin{equation}(V_j-U_j):= \sum_{i=1}^m(V_{i,j}-U_{i,j})^+.\end{equation}

The notation $()^+$ means that only positive differences are summed up. Back to our problem we find $(F_j-A_j)\geq c'_j-c_j$ for all $j$ with $c'_j-c_j\geq 0.$ The reason is that a column $F_j$ has $c'_j$ $1$'s and a column $A_j$ has $c_j$ many. That means the minimum difference of both columns happens if all $1$'s in $A_j$ also occur in $F_j$. Then $(F_j-A_j)= c'_j-c_j.$ Consider in Example~\ref{Example:DifferentDiscrepancies} the third columns of $A$, $B$ and $M.$ $B$ has all $1$'s of $F$ whereas $A$ has only two of them. Hence $F_3-A_3=1>0=c'_3-c_3$ but $F_3-B_3=0=c'_3-c_3.$ In  cases you find a $1$ in $A_j$ which does not occur in $F_j$, this $1$ needs to be shifted to a lefter column $j'$ where $F$ has a $1$ and $A$ a $0.$ A $1$ from a righter column $j''$ in $A$ has to be shifted to $A_j$ to achieve the $c'_j$ $1$'s in $F$. In Example~\ref{Example:DifferentDiscrepancies} this is the case for $A_{6,3}.$  We put this connection in another formula for the discrepancy.  

\begin{proposition}\label{prop:Disc} Given is the $m \times n$ $0,1$-matrix $A$ with non-increasing row sums $(r_1,\dots,r_{m})$ and non-increasing column sums $(c_1,\dots,c_n).$ Let $F$ be the corresponding Ferrers matrix. Then the minimum number of shifts to achieve $F$ from $A$ can be calculated by $disc(A)=\sum_{j \in  \mathbb{N}_n}(F_j-A_j)$.\end{proposition}

\begin{proof}
We prove the claim via induction by the number $n$ of columns. For one column we have $A=F$ and so $disc(A)=0$ and $F_1-A_1=0.$ For two columns we only find shifts from the second two the first column. For each shift there exists an $i$ with $A_{i,1}=0,$ $A_{i,2}=1$ and $F_{i,1}=1,$ $F_{i,2}=0.$ Hence the number of shifts corresponds to $F_1-A_1.$ $F_2-A_2=0$ leads to the expected result. Let us now consider a matrix $A$ with $n$ columns. In column $n-1$ we consider the set $I$ of all indices $i$ where $F_{i,n-1}=1$ and $A_{i,n-1}=0.$ Due to the construction of a Ferrers matrix, and since $F$ and $A$ have equal row sums, we find for all $i \in I$ that $F_{i,n}=0$ and $A_{i,n}=1.$ We construct matrix $B$ by exchanging all entries $A_{i,n-1}=0$ to $1$, and all corresponding entries $A_{i,n}=1$ to $0$ where $i \in I.$ Basically, we apply $|I|$ shifts from the $n$th column to the $(n-1)$th column. Notice that we have $F_n-A_n=0$ due to the construction of $F.$ In matrix $B$ we only find shiftable $1$'s in smaller columns than in column $n$ due to our construction. Furthermore, $F_n-B_n=0.$ We delete the last column of $B$ and get the matrix $B'$ with $n-1$ columns. Matrices $B$ and $B'$ have exactly the same number of minimal shifts because in the first $(n-1)$ columns they are completely identically, and $B$ has in the $n$th column no shiftable $1$'s due to our construction. Furthermore, we know that $A$ has $|I|$ more shifts (in the $n$th column) than $B.$ We apply the induction hypothesis on $B'$ and get 
\begin{align*}
disc(A)-|I|=disc(B)=disc(B')=\sum_{j=1}^{n-1}F_j-B'_j=\sum_{j=1}^{n}F_j-B_j&=&\\
\sum_{j=1}^{n-2}(F_j-A_j)+(F_{n-1}-A_{n-1}-|I|)+(F_n-A_n)=\sum_{j=1}^{n}(F_j-A_j)-|I|.&&
\end{align*}
\end{proof}

We denote an isomorphic element of class $I_A$ by $A^{\sigma}$ where $\sigma:\{1,\dots,n\} \mapsto \{1,\dots,n\}$ is the permutation of the columns of $A$ which exchanges columns with same column sums. Recall that 
permuting columns in $A$ does \emph{not} generate all possible isomorphic versions for $A$. For this we needed to permute rows with equal column sums too. However, this approach covers all possible discrepancies which exist in the whole isomorphic class of $I_A$, because permuting rows with equal sums does not change the number of shifts. 
A permutation $\sigma$ can be divided in the convolution of $k$ permutations $\sigma_i$ if $A$ possesses $k$ different column sums, i.e., $(c_1,c_2,\dots,c_n)=(x_1,\dots, x_1,x_2,\dots,x_2,\dots,x_k,\dots,x_k)$. Each $\sigma_i:\{1,\dots,n\} \mapsto \{1,\dots,n\}$ permutes all columns with column sum $x_i$ independently, and keeps the indices of columns with different sums, i.e., $\sigma=\circ_{i=1}^{k}\sigma_i.$ We denote the set of all possible permutations $\sigma_i$ by $\Sigma_i$, and for all $\sigma$ by $\Sigma.$ For Example~\ref{Example:DifferentDiscrepancies} we get six different matrices which are built by permutations $\sigma_1$ and 
$\sigma_2$ where $\sigma_1$ is the identity permutation and $\sigma_2$ permutes the second, third and fourth columns of $A.$ The next Proposition states that an $A^{\sigma}$ in $I_A$ with minimum discrepancy can be found by calculating the minimum discrepancy in each submatrix of $A$ consisting of equal column sums seperately.

\begin{proposition}\label{Prop:isomDisc}
Given is the $m \times n$ $0,1$-matrix $A$ with non-increasing row sums $(r_1,\dots,r_{m})$ and non-increasing column sums $(c_1,\dots,c_n)=(x_1,\dots, x_1,x_2,\dots,$ $x_2,\dots,x_k,\dots,x_k)$. Let $F$ be the corresponding Ferrers matrix. Then the isomorphic discrepancy of matrix $A$ is 
\begin{equation}\label{eq}
Id(A)= \sum_{i=1}^k\left(\min_{\sigma_i \in \Sigma_i}\sum_{j:c_j=x_i}(F_j-A_{j}^{\sigma_i})\right).
\end{equation}
\end{proposition} 

\begin{proof}
The discrepancy $dis(A^{\sigma})$ of each isomorphic matrix $A^{\sigma}$ can be calculated with Proposition~\ref{prop:Disc}. Since the isomorphic discrepancy is defined as the minimum discrepancy in the set of all isomorphic matrices in $I_A,$ we get 

\begin{equation}
 Id(A)= \min_{\sigma \in \Sigma} \{disc(A^{\sigma})\}= \min_{\sigma \in \Sigma} \left(\sum_{j=1}^n (F_j-A^{\sigma}_j)\right).
\end{equation}

Since each $\sigma$ is the convolution of $k$ permutations $\sigma_i$ which permute only the column indices with same sums ($\sigma=\circ_{i=1}^{k}\sigma_i$) we can rewrite $\left(\sum_{j=1}^n (F_j-A^{\sigma}_j)\right)=\sum_{i=1}^k\left(\sum_{j:c_j=x_i}(F_j-A_{j}^{\sigma_i})\right).$ Together with (4) our claim follows. 
\end{proof}

Proposition~\ref{Prop:isomDisc} says that the minimization of each block with same row sum in $A$ can be done independently. More exactly, we need to find for each block in $A$ with same column sums an order (a permutation $\sigma_i$) of the columns which minimizes the following sum. 

\begin{equation} \label{equ:block} w(\sigma_i):= \sum_{j:c_j=x_i}(F_j-A_j^{\sigma_i}).\end{equation}

We reduce the calculation of $Id(A)$ \emph{for each block} $i$ in $A$ with same column sums $x_i$ to a \emph{minimum weighted perfect matching problem} in a complete bipartite graph $G_i=(F_v,A_v,E_i)$ where $|F_v|=|A_v|.$ Each vertex in $F_v$ corresponds to a column $j$ of Ferrers matrix $F$ and a vertex $A_v$ corresponds to a column $j$ of matrix A if $c_j=x_i.$ For simplicity we denote vertices of $F_v$ by columns $F_j$, or, of $A_v$ by $A_j$, respectively. We assign each edge $\{F_j,A_{j'}\} \in E_i$ the weight $w(F_j,A_{j'}):=(F_j-A_{j'})$. Consider Example~\ref{Example:perfectMatching}.

\begin{example}
\label{Example:perfectMatching}
For matrix $A:=\left(\begin{matrix} 1&1&0&1&0\\0&1&1&0&0\\1&0&0&0&1\end{matrix}\right )$ and its Ferrers matrix $F=\left(\begin{matrix} 1&1&1&0&0\\1&1&0&0&0\\1&1&0&0&0\end{matrix}\right )$ we get the two complete bipartite graphs $G_1$ (with vertices $A_1$,$A_2,$ $F_1$, $F_2$) and $G_2$ (with vertices $A_3,$ $A_4$, $A_5,$ $F_3$, $F_4$, $F_5$). The edge weights between a vertex $A_i$ and $F_j$ are $w(F_j,A_{i}):=(F_j-A_{i}).$ This leads in $G_1$ to a weight $1$ for all edges. Hence, every perfect matching corresponds to an order of the columns in $A$ such that the discrepancy is minimal. Contrary in $G_2$ we find the minimum perfect weighted matching for the edges $\{A_3,F_4\},\{A_4,F_3\},\{A_5,F_5\}$ because $w(F_3-A_4)=0$ but $w(F_3-A_5)=1$ and $w(F_3-A_3)=1.$ All other weights in $G_2$ are $0.$
\end{example}

\begin{theorem}
Calculating the isomorphic distance $Id(A)$ of an $m \times n$ $0,1$-matrix $A$ needs $O(n^3)$ asymptotic time.
\end{theorem}

\begin{proof}
Let $M_i$ be a minimum weighted perfect matching in bipartite graph $G_i.$ Then $\sigma_i:\{1,\dots,n\}\mapsto \{1,\dots,n\}$ with $\sigma_i(j')=j$ for all edges $\{F_j,A_{j'}\} \in M_i$, and $\sigma_i(j)=j$ for all $j$ with $c_j \neq x_i$ is a permutation of all columns in $A$ with sum $x_i$ which minimizes $w(\sigma_i)$ in Equation~(\ref{equ:block}). For $\sigma:=\circ_{i=1}^k\sigma_i $ we can conclude with Proposition~\ref{Prop:isomDisc} that $\sum_{i=1}^k w(M_i)=\sum_{i=1}^k (\sum_{j:c_j=x_i}(F_j-A_j^{\sigma_i}))=disc(A).$
Contrary, each permutation $\sigma$ which minimizes $w(\sigma_1),\dots,w(\sigma_k)$ corresponds in each graph $G_i$ to a perfect matching $M_i$ with $\{F_j,A_{j'}\} \in M_i$ if $\sigma_i(j')=j$ for $F_j \in F_v$ and $A_{j'} \in A_v$ which is by construction minimal.\\
Let $n_i$ be the number of columns with sum $x_i.$ Then the computation of a minimum weighted perfect matching in $G_i$ is $O(n_i^3),$ see for example~\cite{Gab:1974}. Since $\sum_{i=1}^k n_i^3 \leq (\sum_{i=1}^k n_i)^3=n^3$ we get a total asymptotic running time of $O(n^3).$
\end{proof}

Please observe that under certain conditions the construction of several graphs $G_i$ can be avoided. This happens for example when the corresponding Ferrers columns are completely equal like in Example~\ref{Example:perfectMatching} in $G_1$. Then all edge weights in $G_i$ are also equal, and therefore each perfect matching is a solution. Neither is necessary to construct in graphs $G_i$ all vertices of $F_v$. For columns $F_j$ which possess only $0's$ (or only $1$'s) we get for all edges which end in these vertices equal weights $0$ (or $c_j$) like in Example~\ref{Example:perfectMatching} for the fourth and fifth column of $F$. That is we can refine the construction of $G_i$ to $n_i$ vertices in $A_v$ which are connected to all columns $F_j$ with same index but $0<c'_j<m.$ Then the optimal solution is not a perfect but a minimum weighted matching with the size of number of columns $F_j$ in $G_i.$ This would reduce graph $G_2$ to a claw, i.e. $A_3,$ $A_4$ and $A_5$ are connected with $F_3.$ The optimal solution (minimum weighted matching) is then edge $\{A_4,F_3\}.$\\

\end{document}